\documentclass{amsproc}
\usepackage{eurosym}
\usepackage{amssymb}
\usepackage{amsfonts}
\usepackage[utf8]{inputenc}
\usepackage[OT4]{fontenc}

\newcommand{\field}[1]{\mathbb{#1}}
\newcommand{\C}{\field{C}}
\newcommand{\K}{\field{K}}

\newcommand{\R}{\field{R}}

\theoremstyle{plain}
\numberwithin{equation}{section}

\newtheorem{theorem}{Theorem}[section]

\newtheorem{lemma}[theorem]{Lemma}
\newtheorem{definition}[theorem]{Definition}

\newtheorem{remark}[theorem]{Remark}
\newtheorem{example}[theorem]{Example}
\newtheorem{proposition}[theorem]{Proposition}
\newtheorem{corollary}[theorem]{Corollary}

\DeclareMathOperator{\dime}{dim}
\DeclareMathOperator{\graph}{graph}

\def\p{\mathbb{P}}

\makeatletter

\@addtoreset{equation}{section} \makeatother
\author[Zbigniew Jelonek]{Zbigniew Jelonek}
\author[Micha{\l} Laso\'n]{Micha{\l} Laso\'n}

\address[Z. Jelonek]{Institute of Mathematics of the Polish Academy of Sciences, ul.\'{S}niadeckich 8, 00-656 Warszawa, Poland}
\email{najelone@cyf-kr.edu.pl}
\address[M. Laso\'n]{Institute of Mathematics of the Polish Academy of Sciences, ul.\'{S}niadeckich 8, 00-656 Warszawa, Poland}
\email{michalason@gmail.com}

\thanks{Z. Jelonek was supported by Polish National Science Centre grant no. 2013/09/B/ST1/04162. M. Laso{\'n} was supported by the Polish Ministry of Science and Higher Education Iuventus Plus grant no. 0382/IP3/2013/72.}

\date{\today}

\begin{document}

\title[]{Quantitative properties of the non-properness set of a polynomial map}

\begin{abstract}
Let $f$ be a generically finite polynomial map $f: \C^n\to \C^m$ of algebraic degree $d$. Motivated by the study of the Jacobian Conjecture, we prove that the set $S_f$ of non-properness of  $f$ is covered by parametric curves of degree at most $d-1$. This bound is best possible.

Moreover, we prove
that if $X\subset\R^n$ is a closed algebraic set covered by
parametric curves, and $f: X\rightarrow\R^m$ is a generically
finite polynomial map, then the set $S_f$ of non-properness of $f$ is also covered by
parametric curves. Moreover, if $X$ is covered by parametric
curves of degree at most $d_1$, and the map $f$ has degree $d_2$,
then the set $S_f$ is covered by parametric curves of degree at
most $2d_1d_2$.

As an application of this result we show a real version of the Bia\l ynicki-Birula theorem: Let $G$ be a real, non-trivial, connected, unipotent group which acts effectively and polynomially on a connected smooth algebraic variety $X\subset\R^n$. Then  the set $Fix(G)$ of fixed points  has no isolated points.
\end{abstract}

\maketitle

\section{Introduction}

Let $f:X\rightarrow Y$ be a generically finite polynomial map between affine varieties.

\begin{definition}
We say that $f$ is \emph{proper at a point $y\in Y$} if
there exists an open neighborhood $U$ of $y$ such that
$f\vert_{f^{-1}(U)}:f^{-1}  (U)\rightarrow U$ is a proper map. The set of points at which $f$ is not proper is denoted by $S_f$.
\end{definition}


The set $S_f$ was first introduced by the first author in \cite{jel} (see also \cite{jel1,jel21}). It is a good measure of non-properness of the map $f$, and it has interesting applications in pure and applied mathematics \cite{jel3,ha,stas}. The first author proved the following property of the set $S_f$ when the base field is $\C$.

\begin{theorem}[Theorem 4.1 \cite{jel5}]\label{hiper}
Let $X$ be an affine variety over $\C$, and let $f:X\rightarrow\C^m$ be a generically finite polynomial map. If $X$ is $\C$-uniruled (covered by polynomially parametric curves), then the set $S_f$ is also $\C$-uniruled.
\end{theorem}

The first aim of this paper is to give a numerical form of Theorem \ref{hiper}. We introduce the notion of degree of uniruledness and we estimate  this degree in some cases.
In particular if $f$ is a generically finite polynomial map $f: \C^n\to \C^m$ of algebraic degree $d$, we prove that the set $S_f$ of non-properness of the map $f$ is covered by parametric curves of degree at most $d-1$. This bound is best possible.

The second aim of our paper is to  generalize Theorem \ref{hiper} to the field of real numbers (see Theorem \ref{multc1}):
\smallskip
\newline
\emph{Let $X$ be a closed algebraic set over $\R$, and let $f:X\rightarrow\R^m$ be a generically finite polynomial map. If $X$ is $\R$-uniruled,
then the set $S_f$ is also $\R$-uniruled.}
\smallskip

Our third aim  is to prove a real counterpart of  the following theorem of Bia\l ynicki-Birula.

\begin{theorem}[\cite{bial}]\label{bial}
If a connected, unipotent, algebraic group acts on an irreducible affine algebraic variety $X\subset\C^n$, then the set $Fix(G)$ of fixed points of this action has no isolated points.
\end{theorem}

The proof from \cite{bial} is cohomological, and it cannot be extended to the real case. In the last section, as an application of our methods, we modify our approach from \cite{jela} and we give a real counterpart of the result of Bia\l ynicki-Birula (Corollary \ref{glowne3}):
\smallskip
\newline
\emph{Let $G$ be a real, non-trivial, connected, unipotent group
which acts effectively and polynomially on a connected smooth closed
 algebraic variety $X\subset\R^n$. Then the set $Fix(G)$ is
$\R$-uniruled. In particular, it has no isolated points.}

\section{Preliminaries}\label{sec2}

Unless stated otherwise, $\K$ is an arbitrary algebraically closed field (the real field case is explained in Section \ref{real}).
All affine varieties are considered to be embedded in an affine space.

The study of uniruled varieties in projective geometry, that is,
varieties possessing a covering by rational curves, has a long
history. In affine geometry it is more natural to consider
polynomially parametric curves (see the definition below) than rational ones.
Therefore in \cite{jel1} (see also \cite{stas}) the first author
defined $\K$-uniruled varieties as those which are covered by
polynomially parametric curves. In \cite{jela} we refined this definition for
countable fields. In this paper we introduce and study the
corresponding quantitative parameter, the degree of
$\K$-uniruledness.

\begin{definition}
An irreducible affine curve $\Gamma\subset\K^m$ is called a \emph{polynomially parametric curve of degree at most $d$}, if there exists a non-constant polynomial map $f:\K\rightarrow\Gamma$ of degree at most $d$ (by the degree of $f=(f_1,\dots,f_m)$ we mean $\max_i \deg f_i$). A curve is \emph{polynomially parametric} if it is polynomially parametric of some degree.
\end{definition}

We have the following equivalences (see also \cite[Proposition 2.4]{jela}).

\begin{proposition}\label{k-uniruledprop}
Let $X\subset\K^m$ be an irreducible affine variety of dimension $n$, and let $d$ be a constant. The following conditions  are equivalent:
\begin{enumerate}
\item for every  $x\in X$ there exists a polynomially parametric curve $l_x\subset X$ of degree at most $d$ passing through $x$,
\item there exists an open, non-empty subset $U\subset X$ such that for every $x\in U$ there exists a polynomially parametric curve $l_x\subset X$ of degree
at most $d$ passing through $x$,
\item there exists an affine variety $W$ of dimension $\dim X-1$ and a dominant polynomial map $\phi:\K\times W\ni (t,w)\mapsto \phi(t,w)\in X$ such that
$\deg_t \phi \leq d$.
\end{enumerate}
\end{proposition}

\begin{proof}
The implication $(1)\Rightarrow (2)$ is obvious. To prove $(2)\Rightarrow (1)$ suppose that $X=\{x\in\K^m:f_1(x)=0,\dots,f_r(x)=0\}$. For a point $a=(a_1,\dots,a_m)\in\K^m$ and $B=(b_{1,1}:\dots:b_{d,m})\in\p^M$, where $M=dm-1$, let
$$\varphi_{a,b}:\K\ni t\mapsto (a_1+b_{1,1}t+\dots+b_{1,d}^dt^d,\dots,a_m+b_{m,1}t+\dots+b_{m,d}^dt^d)\in\K^m$$
be a polynomially parametric curve.
Note that for every $dm$-tuple $b=(b_{1,1},\dots, b_{m,d})$ we have $\varphi_{a,\lambda b}(t)=\varphi_{a,b}(\lambda t))$ for every $\lambda\in \K^*,$ hence the image of 
$\varphi_{a,b}$  depends only on the class $[b]=B\in \p^M$ but not on $b.$ We will identify $\varphi_{a,b}$ with the curve $\varphi_{a,b}(\K).$

Consider the following  variety and  projection:
$$\K^m\times\p^M\supset V=\{(a,b)\in\K^m\times\p^M:\forall_{t,i}\;f_i(\varphi_{a,b}(t))=0\}\ni(a,b)\rightarrow a\in\K^m.$$
Note that $f_i(\varphi_{a,b}(t))=\sum_k \alpha_{i,k}(a,b)t^k,$ hence the equations $ \{ f_i(\varphi_{a,b}(t))\equiv 0\}$ split into finite number of equations $ \alpha_{i,k}(a,b)=0,$ which are homogeneous with respect to $b.$

From the definition, $(a,b)\in V$ if and only if the polynomially parametric curve $\varphi_{a,b}$ is contained in $X$. Hence the image of the projection is contained in $X$ and contains $U$,
since through every point of $U$ passes a polynomially parametric curve of degree at most $d$. But since the projective space $\p^M$ is complete and $V$ is closed, we find that the image is closed,
and hence it is the whole  $X$.

Let us prove $(2)\Rightarrow (3)$. For some affine chart
$V_j=V\cap \{b_j=1\}$ the above map is dominant. We consider the dominant map
$$\Phi:\K\times V_j\ni (t,\phi)\mapsto\phi(t)\in X.$$
After replacing $V_j$ by  some irreducible component
$Y\subset \K^m$ ($\dime(Y)=s$) the map remains dominant. On an
open subset of $X$ fibers of the map $\Phi'=\Phi\vert_{\K\times
Y}$ are of pure dimension $s+1-n$; let $x$ be one of such points.
From the construction of the set $V$ we know that the fiber
$F=\Phi'^{-1}(x)$ does not contain any line of type
$\K\times\{y\}$, so in particular the image $F'$  of  $F$
under the projection $\K\times Y\rightarrow Y$ (which is a
constructible subset of $Y$) has the same dimension. For a general
linear subspace $L\subset\K^m$ of dimension $m+n-s-1$ the set
$L\cap F'$ is $0$-dimensional (indeed, $F'$ contains an
open and dense subset of $\overline{F'}$). Let us fix such an $L$,
and let $R$ be any irreducible component of $L\cap Y$ intersecting
$F'$. Now the map $\Phi'\vert_{\K\times R}:\K\times R\rightarrow X$
satisfies the assertion, since it has one fiber of dimension $0$
(over $x$) and the dimension of $R$ is $n-1.$ Indeed, in this case we
have dim $\K\times R=\dim X$ and since the fibers of $\Phi'$ have
generically dimension $0$, the map $\Phi'$   has to be
dominant.

To prove the implication $(3)\Rightarrow (2)$ it is enough to notice that for every $w\in W$ the map
$\phi_w: \K\ni t\mapsto\phi(t,w)\in X$ is a polynomially parametric curve of degree at most $d$ or it is constant. The image of $\phi$ contains an open dense subset,
so after excluding the points with infinite preimages (a closed set of codimension at most one) we get an open set $U$ with required properties.
\end{proof}

\begin{definition}\label{k-uniruleddef}
We say that an affine variety $X$ has \emph{degree of $\K$-uniruledness at most $d$} if all its irreducible components satisfy the conditions of Proposition \ref{k-uniruledprop}. An affine variety is called
\emph{$\K$-uniruled} if it has some degree of $\K$-uniruledness.
\end{definition}

To simplify our statements we say that the empty set  has degree of $\K$-uniruledness zero, in particular it is $\K$-uniruled.

\begin{example}
{\rm Let $X\subset\K^n$ be a general hypersurface of degree $d<n.$
It is well-known (see \cite[Exercise V.4.4.3, p. 269]{kol} that
$X$ is covered by affine lines, therefore its degree of
$\K$-uniruledness is one.}
\end{example}

For uncountable (algebraically closed) fields there is also
another characterization of $\K$-uniruled varieties (see \cite[Theorem
3.1]{stas}).

\begin{proposition}\label{k-uniruleduncountable}
Let $\K$ be an uncountable algebraically closed field, and let
$X\subset \K^m$ be an affine variety. The following conditions are
equivalent:
\begin{enumerate}
\item $X$ is $\K$-uniruled,
\item for every  $x\in X$ there exists a polynomially parametric curve $l_x\subset X$ passing through $x$,
\item there exists an open, non-empty subset $U\subset X$ such that for every $x\in U$ there exists a polynomially parametric curve $l_x\subset X$ passing through $x$.
\end{enumerate}
\end{proposition}

\section{The complex field case}
In the whole section we assume that the base field is $\C$. The
condition that a map is not finite at a point $y$ is equivalent to
 it being  locally non-proper in the  topological sense
(there is no neighborhood $U$ of $y$ such that
$f^{-1}(\overline{U})$ is compact). This characterization gives
the following:

\begin{proposition}[\cite{jel}]\label{topo}
Let $f:X\rightarrow Y$ be a generically finite map between affine
varieties. Then $y\in S_f$ if and only if there exists a sequence
$(x_n)$ in $X$, such that $ x_n\rightarrow\infty$ and
$f(x_n)\rightarrow y$.
\end{proposition}

In particular, for a polynomial map $f:\C^n\rightarrow\C^n$, $y\in
S_f$ if and only if either $\dim f^{-1}(y)>0$, or
$f^{-1}(y)=\{x_1,\dots,x_r\}$ is a finite set, but
$\sum_{i=1}^r\mu_{x_i} (f)<\mu(f)$, where $\mu$ denotes 
multiplicity. In other words, $f$ is not proper at $y$ if $f$ is
not  a local analytic covering over $y$.

\begin{theorem}\label{cn}
Suppose $f:\C^n\rightarrow\C^m$ is a generically finite polynomial map of degree $d$. Then the set $S_f$ is covered by parametric polynomial curves of degree at most $d-1$.
\end{theorem}

\begin{proof}
Let $y\in S_f$; by an affine transformation we can assume that
$y=O=(0,0,...,0)\in \C^m$. For the same reason we can assume that
$O\not\in f^{-1}(S_f)$. By Proposition \ref{topo} there exists
a sequence of points $x_k\to\infty$ such that $f(x_k)\rightarrow
O$. Let us consider the line $L_k(t)=tO+(1-t)x_k=(1-t)x_k, \ t\in
\C$. Set $l_k(t)=f(L_k(t))$. Of course we have  deg $l_k\leq d$ for every $k.$ Moreover, we can assume that $\deg l_k>0$,
because infinite fibers cover only a nowhere dense subset of $\C^n$.
Each curve $l_k$ is given by $m$ polynomials of one variable:
$$l_k(t)=(\sum_{i=0}^d a^1_i(k)t^i,\dots,\sum_{i=0}^d a^m_i(k) t^i).$$ Hence 
$l_k$ corresponds to the uniquely determined  point
$$(a^1_0(k),\dots,a^1_d(k);a^2_0(k),\dots,a^2_d(k);\ldots;a^m_0(k),\dots,a^m_d(k))\in \C^N.$$
Since for each $i$, $a^i_0(k)\rightarrow0$ as $k\rightarrow\infty$,
 we can change the parametrization of $l_k$
by setting $t\rightarrow \lambda_k t$ in such a way that $\Vert
l_k\Vert=1$ for $k\gg 0$ (we consider here $l_k$ as an element of
$\C^N$ with Euclidean norm). Now, since the unit sphere is
compact, it is easy to see that there exists a subsequence
$(l_{k_r})$ of $(l_k)$ which  converges to a polynomial map $l
: \C\rightarrow \C^m$ with $l (0) = O$ and  deg $l\leq d.$ Moreover, $l$ is
non-constant, because $\Vert l \Vert = 1$ and $l(0)= O.$ We can
also assume that the limit
 $\lim_{k\rightarrow \infty}\lambda_k=\lambda$ exists in the compactification of the field $\C$. We
consider two cases:
\begin{enumerate}
 \item $\lambda$ is finite: then $L_k(\lambda_kt)=(1-\lambda_kt)x_k\to\infty$ for $t\not=\lambda^{-1}.$
\item $\lambda=\infty$; then $\Vert L_k(\lambda_k t)\Vert \geq
(\vert\lambda_kt\vert-1)\Vert x_k\Vert$, and hence $\Vert
L_k(\lambda_k t)\Vert \to
 \infty$ for every $t\neq 0$.
\end{enumerate}
On the other hand, $f(L_k(\lambda_k t))=l_k(\lambda_kt)\rightarrow
l(t)$; using once more Proposition \ref{topo} this means that the
curve $l$ is contained  $S_f$, and so we see that
$S_f$ is covered by parametric polynomial curves of degree at most $d$.

Now we show that $\deg l<d.$ The idea of the proof is as follows: Note that every curve $l_k$ passes through the point $f(O)$,
but the curve $l=\lim \ l_k$ does not. The reason  is that if $l_k(t_k)=f(O)$, then $ \lim \ t_k=\infty.$
We show that if $\deg l=d$, then we can bound all $t_k$, and consequently we get a contradiction.

Assume that deg $l=d.$ Hence we can assume deg $l_k=d$ for all $k.$
Let
$l(t)=(l_1(t),\dots,l_m(t))$ and
$l_k(t)=(l_1^k(t),\dots,l_m^k(t))$. We can assume that the
component
 $l_1(t)$ has  maximal degree. Denote $f(O)=a=(a_1,\dots,a_m)$. All roots of
 the polynomial $l_1(t)-a_1$ are contained in the interior of some
 disc $D=\{ t\in \C : \vert t\vert<R\}.$ Let $\epsilon =\inf\{ \vert l_1(t)-a_1\vert: t\in
 \partial D\}$. For $k\gg 0$ we have $\vert (l_1-a_1)-(l_1^k-a_1)\vert_D<
 \epsilon.$ Consequently, by the Rouch\'{e}' Theorem these polynomials
 have the same number of  zeros (counted with multiplicities) in $D$. In particular,
 the zeros of $l_1^k-a_1$ are uniformly bounded. All curves $L_k$ pass through $O$, so all $l_k$ pass through $a=f(O)$.
 This means that there is a sequence $t_k$ such that $l_k(t_k)=a$. We have just shown that $|t_k|<R$, since $t_k$ is a root
 of the polynomial $l_1^k-a_1$. So we can assume that the sequence $t_k$ converges to some $t_0$. When we pass to the limit we get $l(t_0)=a$,
 which is a contradiction, since $a=f(O)\not\in S_f.$ Hence deg $l<d.$
\end{proof}

Now let $f:\C^n\rightarrow \C^n$ be a polynomial map with non-vanishing jacobian. The famous Jacobian Conjecture asserts that in this case $f$ is a diffeomorphism (see e.g. \cite{b-c-w,essen}). Despite many efforts the conjecture is still wide open. The main obstruction for its solution is related to the set $S_f$ of non-properness of the map $f$.
Van den Dries and McKenna proved in $1990$ that there is no counterexample to the Jacobian Conjecture for which the set $S_f$ is a union of hyperplanes
(see \cite{dreis}). This suggests that we could solve the Jacobian Conjecture if we had some information about the geometry of the set $S_f$. On the other hand, it is well-known that we can reduce the algebraic degree of the map $f$ to degree $3$ (see \cite{b-c-w,druzk}). The price  we have to pay for this reduction is that in practice we have to consider all possible dimensions, even if we try to solve the problem for a fixed dimension. Theorem \ref{cn} gives the following characterization of the set $S_f$ for generically finite cubic maps $f: \C^n \to \C^n$.

\begin{corollary}\label{glowne1}
Suppose $f:\C^n\rightarrow\C^n$ is a generically finite cubic map. Then the set $S_f$ is covered by lines and parabolas. Moreover, if $f$ is quadratic, then $S_f$ is covered only by lines.
\end{corollary}

\begin{theorem}\label{cxw}
Let  $X=\C\times W\subset \C\times \C^n$ be an affine cylinder and
let $f:\C\times W\ni (t,w)\rightarrow
(f_1(t,w),\dots,f_m(t,w))\in\C^m$ be a generically finite
polynomial map. Assume that $\deg _tf_i\leq
d$ for every $1\le i\le n$.  Then the set $S_f$ has degree of $\C$-uniruledness at most
$d$.
\end{theorem}

\begin{proof} Let $y\in S_f$; by an affine transformation we can assume that $y=O=(0,0,...,0)\in \C^m$. By Proposition
\ref{topo} there exists a sequence  $(a_k,w_k)\in\C\times W$ such that $(a_k,w_k)\to\infty$ and $f(a_k,w_k)\rightarrow y$.
Let us consider the line $L_k(t)=((1-t)a_k, w_k),
\ t\in\C$. Set $l_k(t)=f(L_k(t))$. We can assume that $\deg
l_k>0$, because infinite fibers cover only nowhere dense subset of
$X$. Each curve $l_k$ is given by $m$ polynomials of one variable:
$$l_k(t)=(\sum_{i=0}^d a^1_i(k) t^i,\dots,\sum_{i=0}^d a^m_i(k)
t^i).$$ As before,  $l_k$ corresponds to the single point
$$(a^1_0(k),\dots,a^1_d(k);a^2_0(k),\dots,a^2_d(k);\ldots;a^m_0(k),\dots,a^m_d(k))\in \C^N.$$
Since for each $i$,
$a^i_0(k)\rightarrow0$ as $k\rightarrow\infty$ we can change the parametrization of
$l_k$ by setting $t\rightarrow \lambda_k t$ in such a way that
$\Vert l_k\Vert=1$ for $k\gg 0$ (we consider here $l_k$ as an
element of  $\C^N$ with Euclidean norm). Now, since the unit sphere is
compact, there exists a subsequence $(l_{k_r})$ of $(l_k)$ which
is convergent to a polynomial map $l : \C\rightarrow \C^m$ with
$l(0) = O$. Moreover, $l$ is non-constant, because $\Vert l \Vert
= 1$ and $l(0)= O.$ We can also assume that the limit
 $\lim_{k\rightarrow \infty}\lambda_k=\lambda$ exists in the compactification of the field $\C$. We
consider two cases:
\begin{enumerate}
 \item $\lambda$ is finite; then $L_k(\lambda_kt)=((1-\lambda_kt)a_k,w_k)\to\infty$ for $t\not=\lambda^{-1}.$
\item $\lambda=\infty$; then $\Vert L_k(\lambda_k t)\Vert \geq
\max((\vert\lambda_kt\vert-1)\vert a_k\vert,\Vert w_k\Vert)$, and
$\Vert L_k(\lambda_k t)\Vert \rightarrow\infty$ for every $t\neq
0$.
\end{enumerate}
On the other hand, $f(L_k(\lambda_k t))=l_k(\lambda_kt)\rightarrow
l(t)$; using once more Proposition \ref{topo} we find that the
curve $l$ is contained in  $S_f$, and so
$S_f$ has degree of $\C$-uniruledness at most $d$.
\end{proof}

\begin{corollary}\label{wn}
Let  $f=(f_1,\ldots, f_m) :\C^n\rightarrow \C^m$ be a generically
finite map with  $d=\emph{min}_j\max_i\deg _{x_j}f_i$. Then the
set $S_f$ has degree of $\C$-uniruledness at most $d$.
\end{corollary}

\begin{proof}
Assume that $d=\max_i \deg_{x_1} f_i.$ Then $f: \C\times \C^{n-1}\to \C^m$ and we can apply Theorem \ref{cxw} for 
$W=\C^{n-1}.$
\end{proof}

Let us recall (see \cite{jel5}) that for a generically finite polynomial map $f: X\rightarrow Y$ with $X$ being
$\C$-uniruled the set $S_f$ is also $\C$-uniruled. We have the following ``quantitative''
counterpart of this result:

\begin{theorem}\label{multc}
Let  $X$ be an affine variety with degree of $\C$-uniruledness at
most $d_1$, and let $f:X\rightarrow\C^m$ be a generically finite
map of degree $d_2$. Then the set $S_f$ has degree of
$\C$-uniruledness at most $d_1d_2$.
\end{theorem}

\begin{proof}
By Definition \ref{k-uniruleddef} there exists an affine variety
$W$ with $\dim W = \dim X-1$ and a dominant polynomial map $\phi:
\C\times W\rightarrow X$ of degree at most $d_1$ in the  first coordinate. The equality $\dim  \C\times W = \dim  X$ implies that
$\phi$ is generically finite, hence so is $f\circ\phi:\C\times
W\rightarrow \C^m$, which is of degree at
most $d_1d_2$ in the  first coordinate. By Theorem \ref{cxw}, $S_{f\circ\phi}$ has
degree of $\C$-uniruledness at most $d_1d_2$. We have the inclusion
$S_f\subset S_{f\circ\phi}$, and from Theorem \ref{hiper} we know
that if non-empty,  both sets are of pure dimension $\dim
X-1$, so each component of $S_f$ is a component of $S_{f\circ\phi}$.
This implies the assertion.
\end{proof}

\begin{example}
{\rm Let $f: \C^n\ni (x_1,\ldots, x_n)\mapsto (x_1,
x_1x_2,\ldots,x_1x_n)\in \C^n$. We have $\deg f=2$ and
$S_f=\{x\in\C^n:x_1=0\}$. The set $S_f$ has degree of
$\C$-uniruledness  $1$. This shows that in general Theorem \ref{cn}, Theorem
 \ref{cxw} and Corollary \ref{wn} cannot be improved.}
\end{example}

\begin{example}
{\rm For $n>2$ let $X=\{ x\in \C^n : x_1x_2=1\}$, and $f:X\ni
(x_1,\ldots, x_n)\mapsto (x_2,\ldots, x_n)\in\C^{n-1}$.  The
variety $X$ has degree of $\C$-uniruledness  $1$. Moreover,
$\deg f=1$ and $S_f=\{x\in\C^{n-1}:x_1=0\}$. So the set
$S_f$ has degree of $\C$-uniruledness $1$. This shows that in
general Theorems \ref{cxw} and \ref{multc} cannot be improved.}
\end{example}

\begin{remark}
{\rm By the Lefschetz Principle all the results of this section
remain true for an arbitrary algebraically closed field of
characteristic zero.}
\end{remark}

\section{The real field case}\label{real}

In the whole section we assume that the base field is $\R$. Let us recall that by a \emph{real polynomially parametric curve of degree at most $d$} in a semialgebraic set
$X\subset \R^n$ we mean the image of a non-constant real polynomial map $f:\R\rightarrow X$ of degree at most $d$. 
In general a real polynomially parametric curve need not
  be algebraic, but only  semialgebraic. The real counterpart of Proposition \ref{k-uniruledprop} is the following.

\begin{proposition}\label{r}
Let $X\subset \R^n$ be a closed semialgebraic set, and let $d$ be a constant. The following conditions are equivalent:
\begin{enumerate}
\item
for every  $x\in X$ there exists a polynomially parametric curve $l_x\subset X$ of degree at most $d$ passing through $x$,
\item
there exists a dense subset $U\subset X$ such that for every  $x\in U$ there is a polynomially parametric curve $l_x\subset X$ of degree at most $d$ passing through $x$,
\item
for every polynomial map $f:X\to\R^m$, and every sequence $x_k\in X$ such that $f(x_k)\rightarrow a\in \R^m$ there exists a semialgebraic curve $W$ and a
generically finite polynomial map $\phi:\R\times W\ni (t,w)\mapsto
\phi(t,w)\in X$ such that $\deg_t \phi \leq d$, and there exists a sequence $y_k\in \R\times W$ such that $f(\phi(y_k))\rightarrow a$. Moreover, if
$x_k \rightarrow\infty$, then also $\phi(y_k) \rightarrow\infty.$
\end{enumerate}
\end{proposition}

\begin{proof}
First we prove the implication $(2)\Rightarrow (1)$. Suppose that
$X=\{x\in\R^n:f_1(x)=0,\dots,f_r(x)=0\, g_1(x)\geq 0,...,
g_s(x)\geq 0\}$. For $a=(a_1,\dots,a_n)\in\R^n$ and
$b=(b_{1,1},\dots,b_{d,n})\in\R^M$, where $M=dn$, let
$$\varphi_{a,b}(t)=(a_1+b_{1,1}t+\dots+b_{1,d}^dt^d,\dots,a_n+b_{n,1}t+\dots+b_{n,d}^dt^d)$$
be a polynomially parametric curve. If there exists a polynomially parametric curve of
degree at most $d$ passing through $a$, then after
reparametrization we can assume that it is $\varphi_{a,b}$ for
$\sum_{i,j} b_{i,j}^2=1$. This means that $b\in S_M(0,1)$, where
$S_M$ denotes the unit sphere in $\R^M$. Consider the semialgebraic
set
$$V=\{(a,b)\in\R^n\times S_M(0,1):\forall_{t,i}\;f_i(\varphi_{a,b}(t))=0, \forall_{t,j}\;g_j(\varphi_{a,b}(t))\geq 0\}.$$
The definition of the set $V$ says that for $(a,b)\in V$ the polynomially parametric curve $\varphi_{a,b}(t)$ is contained in $X$.
It is easy to see that $V$ is closed. For any $a\in X$, by the assumption there is a sequence  $a_k\rightarrow a$ such that for every $k$ there is
 a polynomially parametric curve $\varphi_{a_k,b_k}\in V$. We can assume that $\Vert a_k\Vert<\Vert a\Vert+1$ for all $k$. Since
  $V$ is closed and the  sequence $((a_k,b_k))\subset V$
 is bounded,  there is a subsequence $(a_{k_r},b_{k_r})$ which converges to $(a,b)\in V$. Now the polynomially parametric curve $\varphi_{a,b}\subset X$ of degree at most $d$
 passes through $a$.

We prove $(1)\Rightarrow (3)$. Consider the semialgebraic set $V$ as above. We have the surjective map
$$\Phi: \R\times V\ni (t,\varphi_{a,b})\rightarrow\varphi_{a,b}(t)\in X.$$
 Let $f:X\to\R^m$ be a polynomial map,
and suppose $f(x_k)\rightarrow a\in \R^m$ for a  sequence $x_k\in X$. Set
$g=f\circ\Phi$. Hence there exists a sequence $z_k\in \R\times V$
such that $g(z_k)\rightarrow a$. By the curve selection lemma there is a
semialgebraic curve $W_1\subset \R\times V $ such that $a\in
\overline{g(W_1)}$. Set $W_2=p_2(W_1)$, where $p_2: \R\times V
\rightarrow V$ is the projection. If $W_2$ is a curve then let $W:=W_2$, if
it is a point we take as $W$ any semialgebraic curve in $V$ which
contains the point $\pi(W_1)$. Now $W$ and $\Phi\vert_{\R\times W})$ are as required.

Finally, to prove $(3)\Rightarrow (2)$ it is enough to take as $f$ the identity in the third condition.
\end{proof}

\begin{definition}\label{R-uniruleddef}
We say that a closed semialgebraic set $X$ has \emph{degree of $\R$-uniruledness at
most $d$} if it satisfies the conditions of Proposition \ref{r}. A closed semialgebraic set is called \emph{$\R$-uniruled} if it has some degree of
 $\R$-uniruledness.
\end{definition}

\begin{example}
{\rm Let $X=\{ (x,y)\in\R^2 : x\geq 0, y\ge 0\}$. It is easy to
check that the degree of $\R$-uniruledness of $X$ is $2$. It has a
ruling $\{(a, t^2) : a\ge 0\}$.}
\end{example}

Let $X\subset\R^n$ be a closed semialgebraic set, and let $f: X\to\R^m$ be a polynomial map. As in the complex case, we say that it is not proper at a point
$y\in\R^m$ if there is no neighborhood $U$ of $y$ such that $f^{-1}(\overline{U})$ is compact. As before, we denote  by $S_f$
the set of all points $y\in\overline{f(X)}$ at which 
$f$ is not proper. This set is also closed and semialgebraic \cite{jel21}. We have:

\begin{theorem}\label{cn1}
Let $f:\R^n\rightarrow\R^m$ be a generically finite polynomial map of degree $d$. Then the set $S_f$ has degree of $\R$-uniruledness at most $d-1$.
\end{theorem}

\begin{theorem}\label{cxw1}
Let $X=\R\times W\subset \R\times \R^n$ be a closed semialgebraic cylinder and let $f:R\times W\ni (t,w)\mapsto (f_1(t,w),\dots,f_m(t,w))\in \R^m$ be a
generically finite polynomial map. Assume that $\deg _tf_i\leq d$ for every $i$. Then the set $S_f$ has degree of $\R$-uniruledness at most $d$.
\end{theorem}

\begin{corollary}\label{cxw2}
Let $L=\phi(\R)$ be a polynomially parametric curve of degree $D.$ Let $X=L\times W\subset \R\times \R^n$ be a closed semialgebraic cylinder and let $f:L\times W\ni (x,w)\mapsto (f_1(x,w),\dots,f_m(x,w))\in \R^m$ be a
generically finite polynomial map. Assume that $\deg _tf_i\leq d$ for every $i$. Then the set $S_f$ has degree of $\R$-uniruledness at most $dD$.
\end{corollary}

\begin{proof}
For the proof it is enough to note that the mapping $\R\times W\ni (t,w)\mapsto (\phi(t),w)\in L\times W$ is proper nad generically-finite.
\end{proof}

\begin{corollary}\label{wn1}
Let  $f=(f_1,\ldots, f_m) :\R^n\rightarrow \R^m$ be a generically finite polynomial map with $d=\emph{min}_j\max_i\deg _{x_j}f_i$.
Then the set $S_f$ has degree of $\R$-uniruledness at most $d$.
\end{corollary}

The proofs of these facts are exactly the same as in the complex case.
To prove a real analog of Theorem \ref{multc} we need some ideas
from \cite{jela}. Let $X$ be a smooth complex projective surface,
and let $D=\sum_{i=1}^n D_i$ be a simple normal crossing 
divisor on $X$ (we consider only reduced divisors). Let
$\graph(D)$ be the graph of $D$, with vertices $D_i$, and one edge
between $D_i$ and $D_j$ for each point of intersection of $D_i$
and $D_j$.

\begin{definition}
We say that $D$ a simple normal crossing divisor on a smooth surface $X$ is \emph{a tree} if $\graph(D)$ is a tree (it is connected and acyclic).
\end{definition}

The following fact is obvious from graph theory.

\begin{proposition}\label{acykl}
Let $X$ be a smooth projective surface and $D\subset X$ be a
divisor which is a tree. If $D', D''\subset D$ are connected
divisors without common components, then $D'$ and $D''$ have at
most one  point in common.
\end{proposition}

\begin{definition}
Let $X\subset \R^n$ ($X\subset \Bbb P^n$)be an algebraic variety. Hence we have a natural embedding $X\subset \C^n$ ($X\subset \Bbb P^n(C)$). By \emph{the complexification} $X^c$ of the variety $X$ we mean the Zariski closure of $X$ in $\C^n$ ($\Bbb P^n(\C)$).
\end{definition}

Now we are ready to prove a real counterpart of Theorem
\ref{multc}. In particular we show that for a generically finite
map $f: X\rightarrow Y$ of real algebraic sets, the set $S_f$
is also $\R$-uniruled, provided  $X$ is.

\begin{theorem}\label{multc1}
Let  $X\subset \R^n$ be a closed algebraic set with degree of
$\R$-uniruledness at most $d_1$, and let $f:X\rightarrow\R^m$ be a
generically finite polynomial map of degree $d_2$. Then the set
$S_f$ is also $\R$-uniruled. Moreover, its degree of
$\R$-uniruledness is at most $2d_1d_2$.
\end{theorem}

\begin{proof}
Let $a\in S_f$ and let $x_k\in X$ be a sequence of points such
that $f(x_k)\rightarrow a$ and $x_k\to\infty.$ By Proposition \ref{r}
there exists a semialgebraic curve $W\subset \R^Q$  and a generically finite
polynomial map $\phi:  \R\times W \ni (t,w)\rightarrow\phi(t,w)\in X$
such that $\deg_t\phi \leq d_1$, and there exists a sequence
$(y_k)\subset\R\times W$ such that $f(\phi(y_k))\rightarrow a$ and
$y_k\to\infty.$  In particular $a\in S_{f\circ\phi}.$ 
If we  knew that the mapping $\phi$ is proper, then $S_{f\circ\phi}\subset S_f$ and we are done by Theorem \ref{cxw1}.
However, in general it is not true. Our idea is to obtain a suitable compactification $\phi'$ of the map $\phi$, and then to derive all information from the fact that 
$S_{f\circ\phi'}\subset S_{f\circ\phi}$ and  $S_{f\circ \phi'}\subset S_f.$

Let $\Gamma\subset \R^Q$
be the Zariski closure of $W$. We can assume that $\Gamma$ is smooth and irreducible. Denote
$Z:=\R\times\Gamma$. We have the induced map $\phi: Z \rightarrow X$.
Hence we also have  the induced complex map $\phi^c:
Z^c:=\C\times\Gamma^c\rightarrow X^c$, where $Z^c,X^c$ denote the
complexification of $Z$ and $X$ respectively. Note that we can resolve the complex singularities of $\Gamma^c$ and this process
does not affect the real structure of the curve $\Gamma.$ Hence we can assume that $\Gamma^c$ is smooth. 

Let $\overline{\Gamma^c}$ be a smooth completion of $\Gamma^c$ and
let us write $\overline{\Gamma^c}\setminus\Gamma=\{ a_1,...,
a_l\}.$ Let $\p^1\times\overline{\Gamma^c}$ be a projective
completion of $Z^c$. The divisor $D=\overline{Z^c}\setminus
Z^c=\infty\times\overline{\Gamma^c}+\sum_{i=1}^l
\p^1\times\{a_i\}$ is a tree. The map $\phi$ induces a
rational map $\phi:
\overline{Z^c}\dashrightarrow\overline{X^c}$, where
$\overline{X^c}$ denotes the projective closure of $X^c.$ We can
resolve the points of indeterminacy of this map (see e.g., \cite[Theorem 3, p. 254]{sha}):

\smallskip

\begin{center}
\begin{picture}(240,160)(-40,40)
\put(-20,117.5){\makebox(0,0)[l]{$\pi\left\{\rule{0mm}{2.7cm}\right.$}}
\put(0,205){\makebox(0,0)[tl]{$(\overline{Z^c})_m$}}
\put(0,153){\makebox(0,0)[tl]{$(\overline{Z^c})_{m-1}$}}
\put(4,105){\makebox(0,0)[tl]{$\vdots$}}
\put(0,40){\makebox(0,0)[tl]{$\overline{Z}^c$}}
\put(170,40){\makebox(0,0)[tl]{$\overline{X^c}$}}
\put(80,50){\makebox(0,0)[tl]{$\phi$}}
\put(100,140){\makebox(0,0)[tl]{$\phi '$}}
\put(10,70){\makebox(0,0)[tl]{$\pi_1$}}
\put(10,130){\makebox(0,0)[tl]{$\pi_{m-1}$}}
\put(10,180){\makebox(0,0)[tl]{$\pi_m$}}
\put(5,190){\vector(0,-1){30}} \put(5,140){\vector(0,-1){30}}
\put(5,80){\vector(0,-1){33}}
\multiput(20,35)(8,0){17}{\line(1,0){5}}
\put(157,35){\vector(1,0){10}} \put(20,200){\vector(1,-1){150}}
\end{picture}
\end{center}

\smallskip
Note that we can first resolve the real points of indeterminacy. After this process
the variety $H:=\pi^{-1}(\overline{Z})$ still has a structure of a real variety. 
Further we will call all points which are over $\overline{Z}$  real points.
Note that  there is a Zariski open neighborhood 
$U\subset \overline{Z^c}$ of $\overline{Z}$ such that the on $\pi^{-1}(U)$ we have the operation of complex conjugation of points.
Moreover,  $\R\times\Gamma\subset H$. 

Let
$Q:=(\overline{Z^c})_m\cap \phi'^{-1}(X^c).$ Then the map $\phi':
Q\rightarrow X^c$ is proper. Moreover, $Q=(\overline{Z}^c)_m\setminus
\phi'^{-1}(\overline{X^c}\setminus X^c).$ The divisor
$D_1=\phi'^{-1}(\overline{X^c}\setminus X^c)$ is connected  as the
complement of a semi-affine variety $\phi'^{-1}(X^c)$  (for
details see \cite[Lemma~4.5]{jel1}). Note that  the divisor
$D'=\pi^*(D)$ is a tree. Hence the divisor $D_1\subset D'$ is also
a tree.

Note that the map $f'=f\circ\phi'$ is determined on the set $Q^r:=H\cap Q$ and now the mapping $\phi':Q^r\to X$ is proper. 
The mapping $f'$  has a natural extension to the set 
$Q$ and we will consider the regular complex  map $f': Q\rightarrow\C^m.$ 
This map
induces a rational map from $P:=(\overline{Z}_n)^c$ to
$\p^m(\C).$ As before we can resolve its points of indeterminacy:

\smallskip

\begin{center}
\begin{picture}(240,160)(-40,40)
\put(-20,117.5){\makebox(0,0)[l]{$\psi\left\{\rule{0mm}{2.7cm}\right.$}}
\put(0,205){\makebox(0,0)[tl]{$P_k$}}
\put(0,153){\makebox(0,0)[tl]{$P_{k-1}$}}
\put(4,105){\makebox(0,0)[tl]{$\vdots$}}
\put(0,40){\makebox(0,0)[tl]{$P$}}
\put(170,40){\makebox(0,0)[tl]{${\Bbb P^m(\C)}$}}
\put(80,50){\makebox(0,0)[tl]{$f'$}}
\put(100,140){\makebox(0,0)[tl]{$F$}}
\put(10,70){\makebox(0,0)[tl]{$\rho_1$}}
\put(10,130){\makebox(0,0)[tl]{$\rho_{k-1}$}}
\put(10,180){\makebox(0,0)[tl]{$\rho_k$}}
\put(5,190){\vector(0,-1){30}} \put(5,140){\vector(0,-1){30}}
\put(5,80){\vector(0,-1){33}}
\multiput(20,35)(8,0){17}{\line(1,0){5}}
\put(157,35){\vector(1,0){10}} \put(20,200){\vector(1,-1){150}}
\end{picture}
\end{center}

\smallskip

Again we can first resolve the real points of indeterminancy. After this process
the variety $\psi^{-1}(H)$ still has the structure of a real variety. In particular there is a Zariski open neighborhood 
$V\subset P_k$ of $\psi^{-1}(H)$ such that on $V$ we have the operation of complex conjugation of points.

Note that  the divisor $D_1'=\psi^*(D_1)$ is a tree.
Let $\infty'\times\overline{\Gamma}$ denote the proper transform of
$\infty\times\overline{\Gamma}$. It is an easy observation that
$F(\infty'\times\overline{\Gamma})\subset \pi_\infty$, where
$\pi_\infty$ denotes the hyperplane at infinity of $\p^m(\C)$. Now
$S_{f'}= F(D_1'\setminus F^{-1}(\pi_\infty)).$ The curve
$L=F^{-1}(\pi_\infty)$ is connected (by the same argument as above).
Now by Proposition \ref{acykl}  every irreducible
curve $l\subset D'_1$ (note that necessarily $l\cong \p^1(\C)$)
which does is not contained in $L$ has at most one  point in common with
$L$. Let $R\subset S_{f'}$ be an irreducible component. Hence $R$
is a curve. There is a curve $l\subset D_1',$ which has exactly one
 point in common with $L$, such that $R=F(l\setminus L).$ If $l$ is
given by blowing up  a real point, then $L$ also has a real
point in common  with $l$ (because otherwise there are  two conjugate
common points of $l$ and $L$). When we restrict to the real model
$l^r$ of $l$ we have  $l^r\setminus L\cong \R.$ Hence if we
restrict our considerations only to the real points and to the set
$Q^r$,   we see that the set $S$ of non-proper
points of the map $f'|_{Q^r}$ is a union of polynomially parametric curves
$F(l^r\setminus L), \ l\subset D_1', \psi(l)\subset H$. Of
course $a\in S \subset S_f.$ Similarly the set $S_{f\circ\phi}$
 is a union of polynomially parametric curves $F(l^r\setminus
L), \ l\subset \psi^*(D'), \pi(\psi(l))\subset \overline{Z}$. Hence we can say that every ``irreducible
component'' of the set of non-proper points of $f'|_{Q^r}$ is also
an `irreducible' component of $S_{f\circ\phi}$. Moreover $a\in
S_{f'|_{Q^r}}\subset S_f.$ In particular there is a real parametric curve $F(l^r\setminus L)\subset S_f$
 which contains the point $a$ and which is covered by curves lying in  $S_{f\circ \phi}$. Now we can finish the 
 proof by invoking  Theorem
\ref{cxw1} and Lemma \ref{xx} below.
\end{proof}

\begin{lemma}\label{xx}
Let $\psi:\R\to\R^m$ be a polynomially parametric curve. If there exist  polynomially parametric curves $\phi_i: \R\rightarrow\R^m,
\ i=1,...,n$,
of degree at most $d$ with $\psi(\R)\subset \bigcup^n_{i=1} \phi_i(\R)$, then $\psi(\R)$ has degree of $\R$-uniruledness 
at most $2d$.
\end{lemma}

\begin{proof}
Indeed, let $\psi(t)=(\psi_1(t),...,\psi_m(t))$ and let $X$ denote the Zariski closure
of $\psi(\R).$ Consider the
field $L=\R(\psi_1,...,\psi_m)$. By the L\"uroth Theorem
there exists a rational function $g(t)\in \R(t)$ such that
$L=\R(g(t)).$ In particular there exist  $f_1,..., f_m\in \R(t)$
such that $\psi_i(t)=f_i(g(t))$ for $i=1,...,m.$ In fact, we have
two induced maps $\overline{f} : \Bbb P^1(\R)\rightarrow \overline{X}\subset \Bbb
P^m(\R)$ and $\overline{g}: \Bbb P^1(\R)\rightarrow\Bbb P^1(\R).$ Here $\overline{X}$ denotes  the projective
closure of $X.$
Moreover, $\overline{f}\circ \overline{g}=\overline{\psi}.$ Let
$A_\infty$ denote the unique point at infinity 
of $\overline{X}$ and let
$\infty=\overline{f}^{-1}(A_\infty).$ Then
$\#\overline{g}^{-1}(\infty)=\infty$,  i.e., $g\in \R[t].$
Similarly $f_i\in \R[t].$ Now if deg $g=1$ then
$f:\R\rightarrow\R^n$ covers the whole $\phi(\R)$.  Otherwise we can
compose $f$ with a suitable polynomial of degree two to obtain the whole
$\psi(\R)$ as image.

Now let $\phi_i:=\phi$ be a curve which has infinitely many  points in common with $\psi(\R).$
In the same way as above we have $\phi=f'\circ g'$, where $f':\R\to X$ is a birational and polynomial mapping and $g':\R\to \R$ is a polynomial mapping.
In particular $f^{-1}\circ f : \R\to\R$ is a polynomial automorphism, i.e., $f(t)=f'(at+b), \ a\in \R^*, b\in \R.$
Hence we can compose $f'$ with a
suitable polynomial of degree one or two to obtain the whole
$\psi(\R)$ as  image. In any case $\psi(\R)$ has a
parametrization of degree bounded by  $2$ deg $f\le 2d.$
\end{proof}

\begin{corollary}\label{multc3}
Let $X$ be a closed algebraic set which is $\R$-uniruled and let $f:X\rightarrow \R^m$ be a generically
finite polynomial map. Then every connected component of the set $S_f$ is unbounded.
\end{corollary}

\section{An application of the real field case}\label{real2}

As an application we give a real counterpart of a  theorem of Bia\l ynicki-Birula \cite{bial}.

\begin{theorem}\label{glowne}
Let $G$ be a real, non-trivial, connected, unipotent group, which acts effectively and polynomially on a closed algebraic $\R$-uniruled set $X\subset \R^n$.
Then the set $Fix(G)$ of fixed points of this action, is also $\R$-uniruled. In particular, it has no isolated points.
\end{theorem}

\begin{proof} First of all let us recall that a connected unipotent group
has a normal series $$0=G_0\subset G_1\subset\dots\subset G_r=G,$$
where $G_i/G_{i-1}\cong G_a= (\R,+,0).$ By  induction on dim $G$ we
can easily reduce the problem to $G=G_a.$ Indeed,
assume the conclusion holds for $G=G_a.$ Take a unipotent group $G$
with dim $G=n$ and assume that the conclusion holds in dimension $n-1.$ There is a normal subgroup $G_{n-1}$ of dimension
$n-1$ such that $G/G_{n-1}=G_a.$ Moreover, the set
$R:=Fix(G_{n-1})$ is $\R$-uniruled by our hypothesis. Consider the
induced action of the group $G_a=G/G_{n-1}$ on $R.$ The set of
fixed points of this action is $\R$-uniruled and it coincides with
$Fix(G)$.

Hence assume that $G=G_a.$ Let $D$ be the degree of $\R$-uniruledness
of $X.$ Choose  $a \in Fix(G).$ Let $\phi: G\times X\ni
(g,x)\mapsto\phi(g,x)\in X$ be a polynomial action of $G$ on $X.$
This action also induces a polynomial action of the
complexification $G^c=(\C,+)$ of $G$ on $X^c$.  We will denote
this action by $\overline{\phi}.$ Assume that $\deg_g \phi\le d.$
By Definition \ref{R-uniruleddef} it is enough to prove that
there exists a polynomially parametric curve $S\subset Fix(G)$ passing through
$a$ of degree bounded by $dD$. Let $L$ be a polynomially parametric
curve in $X$ passing through $a.$ If it is contained in
$Fix(G)$, then the assertion is true. Otherwise consider a closed semialgebraic
surface $Y=L\times G.$ There is a natural $G-$action on $Y$: for
$h\in G$ and $y=(l,g)\in Y$ we set $h(y)=(l,hg)\in Y.$ Consider the
map
$$\Phi : L\times G\ni (x, g)\rightarrow\phi(g,x)\in X.$$ It is a generically finite
polynomial map. Observe that it is $G$-invariant, which means
$\Phi(gy)=g\Phi(y).$ This implies that the set $S_\Phi$ of points
at which $\Phi$ is not finite is $G$-invariant.
Indeed, it is enough to show that the complement of this set is
$G$-invariant. Let $\Phi$ be finite at $x\in X.$ Then 
there is an open neighborhood $U$ of $x$ such that 
$\Phi : \Phi^{-1}(U)\rightarrow U$ is finite. Now we have the following
diagram:

\smallskip

\begin{center}
\begin{picture}(240,120)(-20,40)
\put(140,160){\makebox(0,0)[tl]{$\Phi^{-1}(gU)=g\Phi^{-1}(U)$}}
\put(20,160){\makebox(0,0)[tl]{$\Phi^{-1}(U)$}}
\put(30,40){\makebox(0,0)[tl]{$U$}}
\put(180,40){\makebox(0,0)[tl]{$gU$}}
\put(190,100){\makebox(0,0)[tl]{$\Phi$}}
\put(40,100){\makebox(0,0)[tl]{$\Phi$}}
\put(95,50){\makebox(0,0)[tl]{$g$}}
\put(95,170){\makebox(0,0)[tl]{$g$}}
\put(33,145){\vector(0,-1){100}} \put(45,35){\vector(1,0){125}}
\put(60,155){\vector(1,0){75}} \put(183,145){\vector(0,-1){100}}
\end{picture}
\end{center}

\smallskip

\noindent This shows that if $\Phi$ is finite
over $U$, then it is finite over $gU.$ In particular this implies
that the set $S_\Phi$ is $G$-invariant. Let  $S_\Phi=S_1\cup ...\cup S_k$ be a decomposition of $S_\Phi$ into polynomially parametric
curves (see  Corollary \ref{cxw2}). Since
 $S_\Phi$ is $G$-invariant,  each 
curve $S_i$ is also $G$-invariant. Note that the point $a$ belongs to $S_\Phi$, because the fiber over $a$ has infinitely many
points. We can assume that $a\in S_1.$ Let us note that  $a$ is also a fixed point for $G^c.$ Let $x\in S_1;$ we want
to show that $x\in Fix(G).$  Indeed, the set $S_1^c$ is also $G^c-$invariant and if $x\not\in Fix(G)$ then  $G^c.x=S_1^c$ and $a$
would be in the orbit of $x$, which is a contradiction. Hence $S_1\subset
Fix(G)$ and we conclude by Theorem \ref{cxw1}.
\end{proof}

\begin{corollary}\label{glowne2}
Let $G$ be a real, non-trivial, connected, unipotent group which acts effectively and polynomially on a closed algebraic set $X\subset \R^n$.
If the set $Fix(G)$ of fixed points of this action, is nowhere dense in $X$, then it is $\R$-uniruled.
\end{corollary}

\begin{corollary}\label{glowne3}
Let $G$ be a real, non-trivial, connected, unipotent  group which acts effectively and polynomially on a connected smooth closed algebraic variety $X\subset\R^n$.
Then the set $Fix(G)$ is $\R$-uniruled.
\end{corollary}

\end{document}